\documentclass[11pt]{article}
\usepackage{microtype}
\usepackage{lineno}
\usepackage{amsmath,amssymb,amsthm}
\usepackage{xspace}
\usepackage[colorlinks = true,
			citecolor = blue,
			linkcolor = blue]{hyperref}
\usepackage[]{todonotes}	

\usepackage{mathtools}  
\usepackage[shortlabels]{enumitem}
\usepackage[style=alphabetic,maxbibnames=100,maxcitenames=10,backend=bibtex]{biblatex}
\renewbibmacro{in:}{}
 
\newcommand{\dd}{\mathop{}\!\mathrm{d}}
\let\del\partial

\let\Re\relax

\DeclareMathOperator{\Re}{Re}

\let\div\relax
\DeclareMathOperator{\div}{div}

\newcommand{\LL}[0]{\boldsymbol L}

\newcommand{\aps}[0]{{\textup{aps}}}
\newcommand{\ess}[0]{\textup{ess}}

\newcommand{\TT}[0]{\boldsymbol T}

\newcommand{\Uper}{U^\per}
\newcommand{\Ulin}[0]{U^\lin}

\usepackage{scalerel,stackengine}
\stackMath
\newcommand\widecheck[1]{%
\savestack{\tmpbox}{\stretchto{%
  \scaleto{%
    \scalerel*[\widthof{\ensuremath{#1}}]{\kern-.6pt\bigwedge\kern-.6pt}%
    {\rule[-\textheight/2]{1ex}{\textheight}}
  }{\textheight}%
}{0.5ex}}%
\stackon[1pt]{#1}{\scalebox{-1}{\tmpbox}}%
}

\newcommand{\coloneq}{\mathrel{\mathop:}=}

\newtheorem{thm}{Theorem}[section]

\newtheorem{lem}[thm]{Lemma}
\newtheorem{prop}[thm]{Proposition}
\theoremstyle{definition}

\theoremstyle{remark}
\newtheorem{rem}{Remark}
\numberwithin{equation}{section}

\newcommand{\ABC}{Albritton--Bru\'e--Colombo}

\DeclareMathOperator{\pr}{Pr}

\newcommand{\lin}[0]{{\textsf{L}}}
\newcommand{\per}[0]{{\textsf{P}}}
\renewcommand{\AA}[0]{\boldsymbol A}
 \newcommand{\KK}{\boldsymbol K}
  \newcommand{\MM}{\boldsymbol M}
   \renewcommand{\SS}{\boldsymbol S}
   \newcommand{\DD}[0]{\boldsymbol D}
\renewcommand{\bar}{\overline}

\DeclareMathOperator{\BS}{BS}

\newcommand{\myt}[1]{\frac1{t^{#1}}}
\newcommand{\ii}{\textup{i}}
\newcommand{\ee}{\textup{e}}

\newcommand{\Lab}{\LL_{\alpha,\beta}}
\newcommand{\lab}{\lambda_{\alpha,\beta}}
\newcommand{\LLab}{\LL'_{\alpha,\beta}}
\newcommand{\llab}{\lambda'_{\alpha,\beta}}
\newcommand{\Tab}[0]{\TT_{\!\alpha,\beta}}
\newcommand{\prab}[0]{\pr_{\alpha,\beta}}
\newcommand{\Da}{\DD_{\alpha}}
\newcommand{\fd}{(-\Delta)}
\newcommand{\fheat}{\ee^{-t\fd^\alpha}}
\newcommand{\fduham}[1]{\mathcal G_{#1}}

\bibliography{abc-paper}

\begin{document}

\title{Nonuniqueness of Leray--Hopf solutions to the forced fractional Navier--Stokes Equations in three dimensions, up to the J. L. Lions exponent }
\author{Calvin Khor, Changxing Miao, and  Xiaoyan Su}
\date{}
\maketitle
\abstract{
In this paper, we show that for $\alpha\in(1/2,5/4)$,  there exists a force $f$ and two distinct Leray--Hopf flows $u_1,u_2$ solving  the forced fractional Navier--Stokes equation starting from rest. This shows that the J.L. Lions exponent is sharp in the class of Leray--Hopf solutions for the forced fractional Navier--Stokes equation. 
\\[0.5em]

}

\section{Introduction}

Consider Leray--Hopf solutions of the fractional Navier--Stokes equation starting from rest, with forcing $f$:
\begin{align}
	\del_t u + u\cdot\nabla u + \nabla p + \fd^{\alpha}u = f, \quad u|_{t=0} = 0 \label{e:ns}
\end{align}
Recall that a solution is Leray--Hopf if it  belongs to the natural energy class $L^\infty(0,T; L^2(\mathbb R^3)) \cap L^2(0,T; H^\alpha(\mathbb R^3))$, (In what follows, we abbreviate $L^q(0,T;X(\mathbb R^3))$ as $L^q_T X$) attains the initial data strongly in $L^2$, and satisfies the energy inequality for a.e. $t\in(0,T]$ and all $s\in[0,t]$,
\[ \|u(t)\|^2_{L^2} +  \int_s^t \|u(t')\|_{H^\alpha }^2 \dd t'  \le \|u(s)\|^2_{L^2}.\]
We prove nonuniqueness of Leray--Hopf solutions, following the recent work of \ABC{} \cite{zbMATH07583008} which proved the case $\alpha=1$. They worked in the `similarity variables' (here adapted to our fractional parabolic scaling)
\begin{align}\begin{gathered}
	\xi = x/t^{1/2\alpha}, \quad \tau = \log t, \quad
	u(x,t)= \myt{1-1/2\alpha}U(\xi,\tau), \\ p(x,t) = \myt{2-1/\alpha}P(\xi,\tau), \quad f(x,t) = \myt{2-1/2\alpha} F(\xi,\tau), 
\end{gathered} \label{e:ss-variables}
\end{align} which transforms \eqref{e:ns} to
\begin{align}
	\del_\tau U + \left(\frac1{2\alpha}-1 - \frac1{2\alpha}\xi\cdot\nabla_\xi \right)U + U\cdot\nabla_\xi U + \fd^{\alpha} U + \nabla P = F. \label{e:ss-vel}
\end{align} These variables are named as such because the powers of $t$ in \eqref{e:ss-variables} are chosen so that a stationary solution to \eqref{e:ss-vel} corresponds to a self-similar solution of \eqref{e:ns}. We can also write the equation for the vorticity $\omega(x,t)=\myt{}\Omega(\xi,\tau)$  in similarity variables, 
\begin{align}
\del_\tau \Omega + \Big({-1} - \frac1{2\alpha}\xi\cdot\nabla_\xi \Big)\Omega + U\cdot\nabla_\xi \Omega - \Omega\cdot\nabla_\xi U + \fd^{\alpha} \Omega  = G.	\label{e:ss-vor}
\end{align}
Here, $G$ is the corresponding forcing, and $U$ is recovered from the vorticity $\Omega$ by inverting the curl via the Biot--Savart operator, i.e. $U=\BS\Omega$ where 
\[\BS\Omega(\xi)= - \int_{\mathbb R^3} \frac {\xi'\times \Omega(\xi-\xi')}{4\pi |\xi'|^3}\dd \xi'.\]

Most of their effort is devoted to the following construction of a compactly supported stationary unstable (forced) Euler vortex $\bar \Omega = \nabla_\xi\times \bar U$. 
Recall that vector fields can be written in cylindrical coordinates $V=V^r e_r + V^\theta e_\theta + V^z e_z.$ Let $L^2_{\aps}$ denote the space of square integrable, divergence-free, axisymmetric pure-swirl vector fields, i.e. 
$$ L^2_{\aps}=\{V\in L^2 : \div V = 0, \ V(r,\theta,z) = V^\theta(r,z)e_\theta\}.$$
\begin{lem}[\ABC{} \cite{zbMATH07583008}]\label{l:abc}
	There exists a vector field $\bar\Omega\in C^\infty_c(\mathbb R^3;\mathbb R^3)$ such that the operator \begin{align}
	-\LL \Omega \coloneq  \bar U\cdot \nabla \Omega - \Omega\cdot \nabla \bar U + U\cdot\nabla \bar \Omega -\bar\Omega \cdot\nabla U
\end{align} with domain 
\begin{align}D(\LL) = \{ \Omega \in L^2_{\aps}: \bar U\cdot\nabla \Omega\in L^2 \}\end{align}
 has an unstable eigenvalue $\lambda$, i.e. an eigenvalue with $\Re \lambda>0$.
\end{lem}

 Their construction is not trivial because the other known unstable flows in 3D like shear flows or 2.5D Euler flows have no decay. They instead built on the recent instability result of Vishik \cite{https://doi.org/10.48550/arxiv.1805.09426, https://doi.org/10.48550/arxiv.1805.09440}, for the forced 2D Euler equations. \ABC{} observed that the 2D Euler equations arises as a formal limit as $r\to\infty$ of the 3D axisymmetric-without-swirl Euler equations. This  allows a perturbative argument to transfer the instability from 2D to 3D.
 
The rest of \ABC's paper turns the linear instability in Euler into nonuniqueness for Navier--Stokes. In this note, we extend this nonuniqueness result to the equation \eqref{e:ns}, where $\alpha\in (1/2,5/4)$. Well-posedness for  \eqref{e:ns} when $\alpha\ge 5/4$ is a well-known result due to J.L. Lions \cite{lions1969quelques}. Our result shows that this is sharp in a stronger way than the non-uniqueness of distributional solutions in \cite{zbMATH07201159}, albeit with a non-zero forcing term.

Following \ABC{}, we proceed as follows. First, we show that for $\beta$ sufficiently large, the linearised operator $\Lab:D(\Lab)\subset L^2_{\aps}\to L^2_{\aps} $ around $\beta \bar\Omega$, defined by
\begin{align}\begin{alignedat}{-1}
D(\Lab)\coloneq \{\Omega\in L^2_{\aps}: \Omega\in H^{2\alpha},\  \xi\cdot\nabla\Omega\in L^2 \},\\
	 -\Lab \Omega \coloneq (-1 - \frac1{2\alpha}\xi\cdot\nabla_\xi + \fd^{\alpha})\Omega - \beta \LL\Omega  ,
\end{alignedat}\label{e:defn-Lab}
\end{align}
has an unstable eigenvalue. This corresponds to an unstable eigenvalue for the operator $\Tab$, which is $ \Lab$ written in the velocity formulation (see \eqref{e:defn-Tab} below). Let $\eta$ be the corresponding eigenvector of $\Tab$, and define $\Ulin=\Re (e^{\lambda t} \eta)$, which solves the linear equation $\del_t \Ulin =  \Lab  \Ulin $. To solve \eqref{e:ss-vel}, we use the ansatz
\[ U = \beta \bar U + \Ulin + \Uper\]
and find $\Uper$ by a fixed point argument in Subsection \ref{ss:nonunique}. Finally, undoing the similarity variable transformation proves the following:
\begin{thm}\label{t:main}
Let $\alpha\in(1/2,5/4)$ and  $\bar U=\BS \bar\Omega$ be the smooth unstable velocity from Lemma \ref{l:abc}. Set $\tau = \log t $, $\xi = xt^{-1/2\alpha}$. Then for $\tau \le T$, 
 $\bar u(x,t)\coloneq\frac\beta{t^{1-1/2\alpha}} \bar U(\xi,\tau)$ and $u(x,t)\coloneq \frac1{t^{1-1/2\alpha}}  U(\xi)$, where $U$ is constructed in Subsection \ref{ss:nonunique}, are two distinct Leray--Hopf solutions to \eqref{e:ns} with force $f(x,t) = \frac1{t^{2-1/2\alpha}} F(\xi,\tau)$ on a time interval $[0,e^T]$, where 
\[ F\coloneq \beta \left(\frac1{2\alpha}-1 - \frac1{2\alpha}\xi\cdot\nabla_\xi \right)\bar U + \beta^2 \bar U\cdot\nabla_\xi \bar U + \beta \fd^{\alpha} \bar U. \]  
In addition, $\bar u$ and $u$ belong to the borderline space $L^\infty_T L^{\frac3{2\alpha-1},\infty}$ and for any $k,j\ge0 $, $p\in[2,\infty]$ and $ t< e^T$,
\begin{equation}
t^{k/2}\|\bar u(t)\|_{\mathring W^{k\alpha,p}}+	t^{k/2}\|  u(t)\|_{\mathring W^{k\alpha,p}} + 	t^{j+1+k/2}\| \del_t^jf(t)\|_{\mathring W^{k\alpha,p}}
 \lesssim_{k} t^{\frac{p+3}{2\alpha p} -1}, \label{e:scaling}
\end{equation}
where we have written $\|v\|_{\mathring W^{s,p}} \coloneq \|\fd ^{s/2} v\|_{L^p}$.
\end{thm}

\begin{rem} 
The only place that $\alpha<5/4$ is used is in the last step  to attain the initial data via \eqref{e:scaling}. The proof also extends to any subcritical diffusion just below the $5/4$ exponent, for instance a logarithmically modified version of $\fd^{5/4}$. The condition $\alpha>1/2$ ensures that the diffusion is stronger than the effects from the material derivative, which we use in Lemma \ref{lem:unstable-vel-op-Tab}, Lemma \ref{lem:rel-cpct}, and Proposition \ref{p:contraction} below.\end{rem}
%
\begin{rem}
	The methods used in Vishik  and \ABC{}, and hence here, strongly rely on the equation having a forcing term that we can choose. In addition to relying on Vishik's unstable vortex, which only solves a forced Euler equation, this gives us the flexibility to linearise around vector fields that do not solve the unforced equations.
\end{rem}
\begin{rem}
Recently, Albritton and Colombo have shown \cite{https://doi.org/10.48550/arxiv.2209.08713} non-uniqueness for the hypodissipative (i.e. $\alpha < 1$ in our notation)  forced 2D Navier--Stokes equation, as announced in \cite{zbMATH07583008}. In addition, Albritton, Bru\'e and Colombo have recently shown \cite{https://doi.org/10.48550/arxiv.2209.03530} that the result of \cite{zbMATH07583008} extends to $\mathbb T^3$ and bounded domains of $\mathbb R^3$ when the no-slip boundary condition is enforced. One can also see the recent preprint \cite{https://doi.org/10.48550/arxiv.2209.12196} which shows well-posedness for sufficiently small forces in various senses, and initial data in the critical space $\textup{BMO}^{-1}$.
\end{rem}

\section{Linear instability}
To replace the heat semigroup estimates used in \cite{zbMATH07583008}, we need to recall the following basic estimate for the fractional heat equation. For given functions $u_0=u_0(x)$ and $f=f(x,t)$, we write $u=\fheat u_0$ for the solution to 
\[\del_tu + \fd^\alpha u= 0 ,\qquad u|_{t=0} = u_0, \]
and write $u=\fduham 0f$ where $\fduham0 f = \int_0^t \ee^{-(t-s)\fd^\alpha} (f(\cdot,s)) \dd s  $ is the solution to 
\[\del_tu + \fd^\alpha u= f ,\qquad u|_{t=0} = 0. \]
Then, the solution of the general initial value problem
\[\del_tu + \fd^\alpha u= f ,\qquad u|_{t=0} = u_0 \]
is the sum $u=\fheat u_0 + \fduham 0f.$
\begin{lem}
\label{lem:fracheat-est}
\cite[Lemma 3.1]{zbMATH05222116}  For $ 1 \le r \le p \le \infty$ and $\beta \ge 0$, \[\| \fd^{\beta/2}  \fheat u_0\|_{L^p} \lesssim t^{-\frac{3}{2\alpha}(\beta + \frac1r - \frac1p)}\|u_0\|_{L^r}. \]
\end{lem}

\begin{lem} \label{lem:densely-def-Lab}
$\Lab$ with the domain $D(\Lab)$ as in \eqref{e:defn-Lab} is a closed, densely defined operator.
\end{lem}
\begin{proof}
If $\Omega_n\in D(\Lab)$ converges to $\Omega$ and $\Lab\Omega\in L^2$ converges to $W$, then $W=\Lab \Omega$ in the sense of distributions. Since $\LL$ is closed, in order to show that $\Lab$ is closed, we only  need to show that each $Z\in L^2$ corresponds to a unique solution $\Omega \in H^{2\alpha}$ to the equation $(-\Lab +1 + \beta \LL )\Omega = Z$, i.e.
\[ -\frac1{2\alpha}\xi\cdot\nabla \Omega +\fd^{\alpha}\Omega = Z. \]
We deal with the unbounded term here and in the sequel by using the fact that transforming back to physical variables by $\xi = x/t^{1/2\alpha }$ converts $\xi\cdot\nabla_\xi $ into a time derivative. Specifically, we consider the functions
\begin{align}
	h(x,t)=\Omega(x/t^{1/2\alpha}), \quad g(x,t)= \frac1tZ(x/t^{1/2\alpha}).\label{e:heat-transform}
\end{align}
It is easy to check that $h$ solves the fractional heat equation 
\[ \del_t h + \fd^{\alpha} h  = g,\]
and attains zero initial data in $L^2$ since $\alpha>0$:
\[ \|h\|_{L^2} = t^{3/4\alpha } \|\Omega\|_{L^2} \xrightarrow[t\to0]{}0.\]
Since $\|g\|_{L^2}=t^{-1+3/4\alpha } \|Z\|_{L^2}$ is integrable, we obtain via the theory of the fractional heat equation 
 that $\Omega=h(\cdot,1)\in H^{2\alpha}$. Indeed, firstly, $h(\cdot,1/2)$ is $L^2$ by Lemma \ref{lem:fracheat-est}. Then for times $t\ge 1/2$, by writing 
\[ h(x,t) = \fheat h(x,1/2) + \fduham{0}(g(\cdot,t+1/2))(x,t-1/2), \]
since $g(\cdot,t+1/2)$ is not singular at $t=0$, Lemma \ref{lem:fracheat-est} shows that $\fd^\alpha h \in C([1/2,1];L^2)$, so we obtain the claimed result. Hence, $\Lab$ is closed.
\end{proof}

We write  $\LLab =\frac1\beta \Da+ \MM+\KK+ \SS $ as a dissipative term, main term, compact term, and small term respectively, where:
\begin{align*}
	-\Da \Omega &\coloneq -\frac{3}{4\alpha}\Omega - \frac1{2\alpha}\xi\cdot\nabla \Omega + \fd^{\alpha} \Omega ,\\
	-\MM\Omega&\coloneq \bar U \cdot\nabla\Omega ,\\
	-\KK\Omega&\coloneq U\cdot\nabla\bar \Omega ,\\
	-\SS\Omega&\coloneq \bar \Omega \cdot\nabla U + \Omega \cdot\nabla \bar U.
\end{align*}
Then $\Lab = \beta \LLab + 1 - \frac3{4\alpha } $. The constant $1-\frac3{4\alpha}$ is added because 
$\xi\cdot\nabla \Omega$ is not skew-adjoint.

\begin{lem}[Resolvent estimates via inviscid limit]
For all $\beta>0$ and all $\lambda$ with $\Re\lambda>\mu\coloneq \|\SS\|_{L^2_\aps \to L^2}$, 
\begin{align}
    \|R(\lambda,\beta^{-1}\Da +\MM + \SS)\|_{L^2_{\aps}\to L^2_{\aps}} \le \frac1{\Re \lambda-\mu},\label{e:resolvent-bd}
\end{align}
and the resolvent of $\beta^{-1}\Da +\MM + \SS$ converges in the strong topology, i.e. for all $\Omega_0\in L^2_{\aps}$,
\[R(\lambda, \beta^{-1}\Da +\MM + \SS) \Omega_0 \xrightarrow[\beta\to\infty]{} R(\lambda, \MM + \SS)\Omega_0. \]
\end{lem}
\begin{proof} 
We want to use that the Laplace transform of the semigroup is the resolvent (see \eqref{e:resolvent-laplace} below), so let $\Omega^\beta$ be the solution to the initial value problem $(\del_\tau - \beta^{-1}\Da-\MM-\SS)\Omega =0$, i.e. 
\[\left\{\begin{alignedat}{-1}\del_\tau \Omega^\beta - \frac1\beta \Big(\frac3{4\alpha}+\frac1{2\alpha}\xi\cdot\nabla -\fd^{\alpha}\Big)\Omega ^\beta + \bar U \cdot\nabla \Omega ^\beta  &=  \SS\Omega ^\beta,\\ \Omega ^\beta |_{\tau=0}&=\Omega_0,\end{alignedat}\right.\]
and let $\Omega$ solve the $\beta\to\infty$ limit equation
\[\left\{\begin{alignedat}{-1}\del_\tau \Omega + \bar U \cdot\nabla \Omega &=  \SS\Omega ,\\ \Omega  |_{\tau=0}&=\Omega_0,\end{alignedat}\right.\]
 For \emph{smooth, compactly supported} data, $\Omega$ exists and is smooth. $\Omega^\beta$ can be shown to exist as a strong solution in $C^0 H^{2\alpha}$ by again undoing the similarity variable transform which removes the term $\xi\cdot\nabla$ to get a more standard fractional heat equation with drift, similarly to \eqref{e:heat-transform}.
Specifically, one sets $h^\beta(x,t)=\Omega^\beta (xt^{-1/2\alpha},\beta \log t)$ and  $f^\beta (x,t)=\beta t^{-1}\SS\Omega^\beta(xt^{-1/2\alpha},\beta \log t)$ to get
\[ \del_t h^\beta  -\left(\frac3{4\alpha} - \fd^\alpha \right ) h^\beta + \beta \bar u\cdot \nabla h^\beta  = f^\beta. \]
Notably, $\bar u$ is smooth in both space and time, as the initial data is prescribed at $t=1$, so well-posedness is easily proven first for $f^\beta=0$, then for $f^\beta\neq 0$ with Duhamel's formula.

From the energy inequality for $\Omega^\beta$, and from the Lagrangian formulation for $\Omega$, we have the  estimates
\begin{align}
\|\Omega\|_{L^2}^2  + \|\Omega^\beta\|_{L^2}^2 \lesssim \ee^{\tau\smash { \|\SS\|_{L^2_{\aps}\to L^2}} }\|\Omega_0\|^2_{L^2}=e^{\tau \mu }\|\Omega_0\|^2_{L^2}.\label{e:growth-bound}
\end{align}
For the  limit, we study the equation for the difference $\tilde\Omega\coloneq \Omega^\beta - \Omega$,
\[(\del_\tau -\beta^{-1} \Da -\MM -\SS)\tilde\Omega = \beta^{-1}\Da \Omega,\] initially for  $\Omega_0\in C^\infty_c$. Testing against $\tilde\Omega$ in $L^2$, and using Young's inequality ($ab\le a^2/2 + b^2/2$) gives
\begin{multline}
	\frac{\dd}{\dd\tau} \|\tilde\Omega\|_{L^2}^2 + \frac1\beta \|\fd^{\alpha/2}\tilde\Omega\|_{L^2}^2 \le \mu \|\tilde \Omega \|_{L^2}^2 + \frac1\beta \langle\Da\Omega,\tilde \Omega\rangle_{H^{-\alpha},H^\alpha }
\\
\le \mu\|\tilde\Omega\|_{L^2}^2 + \frac 1{2\beta} \|\Da\Omega\|_{H^{-\alpha}}^2 + \frac1{2\beta}\Big( \|\tilde\Omega\|_{L^2}^2 +\|\fd^{\alpha/2} \tilde\Omega \|_{L^2}^2\Big).
\end{multline}
Absorbing $\frac1{2\beta}\|\fd^{\alpha/2} \tilde\Omega \|_{L^2}^2$ into the left-hand side, integrating in $0\le \tau \le T$ and using that $\tilde\Omega|_{\tau=0}=0$, we obtain for each fixed $T$ that 
\begin{align}
    \|\tilde\Omega\|_{L^\infty(0,T;L^2)}^2 \lesssim_T \beta^{-1}\sup_{t\le T} \|\Da\Omega(\cdot,t)\|_{H^{-\alpha}}^2 \xrightarrow[\beta\to\infty]{}0. \label{e:inviscid-limit}
\end{align}
Recall that for a densely defined operator $\AA$ with sufficient growth bounds on the semigroup $e^{t\AA}$, the resolvent $R(\lambda,\AA)$     can be written as the Laplace transform 
\begin{align}
     R(\lambda,\AA)\Omega_0 = \int_0^\infty \ee^{-t (\lambda-\AA)} \Omega_0 \dd t. \label{e:resolvent-laplace}
\end{align}
In our case, this formula and the growth bound \eqref{e:growth-bound} shows the boundedness \eqref{e:resolvent-bd}.
It follows that for all $\lambda$ with $\Re\lambda>\mu$, \begin{align*}
	 \MoveEqLeft\|R(\lambda,\beta^{-1}\Da + \MM + \SS)\Omega_0 -  R(\lambda,  \MM + \SS)\Omega_0\|_{L^2}  \\&\le \int_0^\infty \ee^{-\Re \lambda t} \|
	\ee^{t(\beta^{-1}\Da + \MM + \SS)} \Omega_0 - \ee^{t( \MM + \SS)} \Omega_0  \| \dd t \\
	 &\le \|\tilde\Omega\|_{L^\infty(0,T; L^2)}\int_0^T\ee^{-\Re \lambda t}\dd t + \int_T^\infty \ee^{-\Re \lambda t}(\|\Omega(t)\|_{L^2} + \|\Omega^\beta(t)\|_{L^2}) \dd t
	 \\
	 	 &\le  \frac1{\Re\lambda} \|\tilde\Omega\|_{L^\infty(0,T; L^2)}e^{-\Re\lambda T} + \frac1{\Re\lambda-\mu}\|\Omega_0\|_{L^2}e^{-(\Re\lambda-\mu)T}.\end{align*} 
The strong convergence now follows from \eqref{e:inviscid-limit} for all compactly supported data, by first taking $T\gg1$ and then $\beta\to \infty$. We can then cover the case  $\Omega_0\in L^2$ by density. \end{proof}
\begin{lem}
	Let $\lambda_\infty$ be an unstable eigenvalue of $\LL$. Then, for each $\epsilon\in (0,\Re\lambda_\infty - \mu)$, there is $\beta_0$  such that for all $\beta>\beta_0$, $\Lab$ has an unstable eigenvalue $\lab=\beta \llab+(1-\frac3{4\alpha})$, with $|\lambda_\infty - \llab|<\epsilon$. 
\end{lem}
\begin{proof}

	To prove this, we let $\gamma\subset \rho(\LL)\cap \{\Re\lambda>\mu \}$ be a sufficiently small circle around the unstable eigenvalue $\lambda_\infty$ which encloses no other part of $\sigma(\LL)$. We need to show that we can define the spectral projection operator
	\begin{align}
	     \prab \coloneq \frac1{2\pi i} \int_{\gamma} R(\lambda,\LLab)\dd \lambda. \label{e:prab}
	\end{align}
	and show it converges in norm to the corresponding spectral projection for $\LL=\MM+\SS+\KK$,
		\begin{align}
	     \pr \coloneq \frac1{2\pi i} \int_{\gamma} R(\lambda,\LL)\dd \lambda. \label{e:pr}
	\end{align}
	From Lemma \ref{l:abc}, $\pr$ is non-trivial, so $\prab$ is nontrivial for sufficiently large $\beta$. This means that $\gamma$ encloses an eigenvalue of $\Lab$ with finite multiplicity, which proves the Lemma.
	
	So we just need to justify \eqref{e:prab}. Consider the identity
	\[ \lambda - \LLab = (\lambda - \beta^{-1}\Da-\MM-\SS)\big(I-R(\lambda , \beta^{-1} \Da +\MM+\SS)\KK \big),\]   
	which makes sense for $\Re\lambda>\mu$ by \eqref{e:resolvent-bd}. As $R(\lambda , \beta^{-1} \Da +\MM+\SS)\to R(\lambda , \MM+\SS)$ in the strong topology and $\KK$ is compact, we have the norm convergence
	\[ R(\lambda , \beta^{-1} \Da +\MM+\SS) \KK \xrightarrow[\beta\to\infty ]{} R(\lambda , \MM+\SS)\KK, \]
	locally uniformly in $\lambda$.
	In addition,  for $\lambda\in \rho(\LL)\cap \{\Re\lambda>\mu \} \subset \rho(\MM+\SS)$, the identity
	\[ \lambda-\LL=\lambda -(\MM+\SS+\KK) = (\lambda- \MM+\SS)(I-R(\lambda,\MM+\SS)\KK)\]
	shows that $I-R(\lambda,\MM+\SS)\KK$ is invertible. As the set of invertible operators is open, 
	it follows that $I-R(\lambda , \beta^{-1} \Da +\MM+\SS)\KK $ is invertible for $\beta\gg1$, locally uniformly in $\lambda$. Hence, $R(\lambda, \LLab)$ for $\beta\gg1$ makes sense on the compact set $\gamma$, which verifies \eqref{e:prab}. This finishes the proof. 
\end{proof}

\begin{lem}\label{lem:unstable-vel-op-Tab}
For $\beta\gg1$, the linearised Navier--Stokes velocity operator in similarity variables,
\begin{align}
\begin{alignedat}{-1}
	\Tab&: D(\Tab)\subset L^2\to L^2 , \\   D(\Tab)&\coloneq \{ U\in H^{2\alpha} : \xi\cdot\nabla U\in L^2\} ,
	 \\
	-\Tab U &\coloneq	 \Big(\frac1{2\alpha}-1 - \frac1{2\alpha}\xi\cdot\nabla_\xi \Big)U + \beta \mathbb P(\bar U\cdot\nabla_\xi U + U\cdot\nabla_\xi \bar U)  \\&\quad  + \fd^{\alpha} U ,
\end{alignedat}\label{e:defn-Tab}
\end{align}
has the same unstable eigenvalue $\lab $ as $\Lab$. Here, $\mathbb P$ denotes the usual Leray projector to the space of divergence-free fields.
\end{lem}
\begin{proof}
	Algebraically, this is true because the operators in velocity and vorticity formulation are linked by conjugation with the curl operator,
	\[\nabla\times \Tab V = \Lab \nabla\times V. \]
	Hence, $\Tab$ has the same unstable eigenvalue as $\Lab$, and if $\zeta$ is the eigenfunction for $\Lab$, then $\eta=\BS\zeta$ is formally the eigenfunction for $\Tab$. The only detail to check is that $\eta \in D(\Tab)$, and the conditions defining $D(\Tab)$ are all satisfied save for potentially $\eta\in L^2$. As $\BS$ is bounded from $L^2$ to $L^6$, and from $L^1$ to $L^{3/2}$, it suffices by interpolation to show that $\zeta\in L^1$. 

As an eigenfunction of $\Lab$, $\zeta$ satisfies the equation
\[ \lab \zeta - \Big( 1 + \frac1{2\alpha}\xi\cdot\nabla - \fd^\alpha \Big)\zeta = \LL \zeta. \]
We again undo the similarity variables by setting $\xi = x/t^{1/2\alpha}$ and 
\[ h(x,t) \coloneq t^{\lab -1} \zeta \Big(\frac x{t^{1/2\alpha}}\Big), \quad g(x,t) \coloneq t^{\lab-2} (\LL\zeta)\Big(\frac x{t^{1/2\alpha}}\Big).\]
The corresponding equation is
\[\del_t h + \fd^\alpha h =g, \quad h|_{t=0} =0,\]
where the initial data is attained in $L^2$ once $\Re \lab > 1-\frac3{4\alpha}$, using $\beta\gg1$. Observe that $\LL \zeta\in L^2$ is compactly supported. It follows from the integral formulation of the Duhamel-type operator $\mathcal G_0$ that $g\in C^0([0,T]; L^1)$ for each $T$ (again using $\beta\gg 1$). Lemma \ref{lem:fracheat-est}  shows that $h$ is also $C^0L^1$, so in particular $\zeta = h(\cdot,1)\in L^1$, as needed.
\end{proof}

\section{Nonlinear construction}

\subsection[The semigroup generated by Talphabeta]{The semigroup generated by $\Tab$}
Here, we adapt the arguments of \cite[\S 2]{zbMATH06433355} to prove some results for the spectrum and the semigroup generated by $\Tab$. Since $\Tab-\TT_{\alpha,0}$ will turn out to be  $\TT_{\alpha,0}$-compact, we first study the spectrum of $\TT_{\alpha,0}$.

\begin{lem}\label{lem:spectrum-TTa0}
 $\sigma (\TT_{\alpha,0}) \subset \{\Re \lambda \le  1 - \frac5{4\alpha}\}$.    
\end{lem}
\begin{proof}
 Let $\Re \lambda > 1- \frac5{4\alpha}.$ To show that $\lambda\in \rho(\TT_{\alpha,0})$, we need to show that $(\lambda - \TT_{\alpha,0})U = F$ can be solved for $U$ with $\|U\|_{H^{2\alpha}} \lesssim \|F\|_{L^2}$. This will be proven by once again undoing the similarity transform, as in the proof of Lemma \ref{lem:densely-def-Lab}. Suppose we are given $U$. The definition \eqref{e:defn-Tab} leads us to define
 \[  h(x,t) \coloneq t^{\lambda - 1 + \frac1{2\alpha}} U\left(\frac{x}{t^{1/2\alpha}}\right), \quad g(x,t) = t^{\lambda - 2 + \frac1{2\alpha}} F\left(\frac{x}{t^{1/2\alpha}}\right). \]
 They solve $\del_t h + \fd^\alpha h = g$ attaining the data $h|_{t=0} = 0$ strongly in $L^2$, since $\Re\lambda > 1 - \frac5{4\alpha}  $. Since $g \in C^0((0,T]; L^2)$, it follows that $h\in C^0((0,T]; H^{2\alpha})$, in particular $U = h(\cdot,1) \in H^{2\alpha}$ with 
 $\|\fd^{2\alpha} U\|_{L^2} \lesssim \|F\|_{L^2}.$  Read backwards, the above explains how to construct $U$ given the force $F$, using the well-posedness of the fractional heat equation. This shows that $\{\Re\lambda > 1 - \frac5{4\alpha} \} \subset \rho(\TT_{\alpha,0})$, hence the result. 
\end{proof}

\begin{lem} $\KK_{\alpha,\beta}$ defined by 
$\KK_{\alpha,\beta} U \coloneq \Tab U-\TT_{\alpha,0} U = \beta \mathbb P ( \bar U\cdot\nabla U+ U\cdot\nabla \bar U)$ is $\TT_{\alpha,0}$-compact.    \label{lem:rel-cpct}
\end{lem}
\begin{proof}
By the estimates in the proof of Lemma \ref{lem:spectrum-TTa0}, the sequence $U_n $ belongs to $D(\Tab)$ with $\TT_{\alpha,0} U_n$ bounded in norm, and is in particular bounded in $H^{2\alpha}$. Hence, the relative compactness follows by the compactness of the support of $\bar U$, which allows the use of the compact embedding of Sobolev spaces on bounded domains. 
\end{proof}
\begin{lem}
The semigroup $\ee^{\tau \Tab}$ is well-defined and strongly continuous. It is the solution operator for the initial value problem
\[ \del_\tau U - \Tab U = 0,\quad U|_{\tau = 0} = U_0 \]
\end{lem}
\begin{proof}
    We undo the similarity variables to find the more standard problem    \[ \del_t u  + \beta \mathbb P(\bar u \cdot\nabla u + u \cdot\nabla \bar u) + \fd ^\alpha u =0, \quad u|_{t = 1} = U_0.\]
         \[\|u(t)\|_{L^2} + (t-1)^{k/2\alpha} \|\nabla^{k} u(t)\|_{L^2} \lesssim_{\bar U,T} \|u(1)\|_{L^2},\]
        which translate to the following estimates for $U$,
\begin{align}
    \|U(\tau )\|_{L^2} + \tau ^{k/2\alpha } \|\nabla ^{k} U(\tau )\|_{L^2} \lesssim_{\bar U,T} \|U_0\|_{L^2}.  \label{e:local-estimates}
\end{align}
    
    This well-defines the solution operator $S_t U_0 = U(\cdot,t)$ which is generated by some closed operator $\AA$ densely defined on some domain $D(\AA)$, and we need to show that $\AA = \Tab$. 
    
    $D(\AA)$ contains the subspace of divergence-free Schwartz fields, and $\AA=\Tab$ for such fields. 
    So, we are done if we show that $D(\Tab)\subset D(\AA).$
    Let $V\in D(\Tab)$, $\lambda \in \rho(\Tab)$  and let $V_n$ be a sequence of divergence-free Schwartz fields such that $V_n\to V$ in $H^{2\alpha}$. Then $(\Tab-\lambda) V_n \to (\Tab-\lambda) V$ in $L^2$. But $\AA$ is closed and agrees with $\Tab$ along the sequence $V_n$, so $V\in D(\AA)$.
    \end{proof}
    For a closed, densely defined operator $\AA$ that generates a semigroup $e^{\tau \AA}$, recall from \cite{zbMATH01354832} the following definitions:
    \begin{align*}
     \omega_{0} (\AA)&\coloneq \inf_{\tau>0} \frac1\tau \log \big\| \ee^{\tau \AA } \big\|, \\ 
         \omega_{\ess} (\AA)&\coloneq \inf_{\tau>0} \frac1\tau \log \inf\big \{  \big\| \ee^{\tau \AA } - K\big\|:K \text{ is compact} \big\}.
    \end{align*}
\begin{lem}
$
\omega_{\ess} (\Tab) \le  1 - \frac5{4\alpha}.
$
\end{lem} 
\begin{proof}
By the relative compactness proved in Lemma \ref{lem:rel-cpct}, we have
\[ \omega_{\ess}(\Tab) = \omega_{\ess}(\TT_{\alpha,0}) \le \omega_0 (\TT_{\alpha,0}),\]
and $\omega_0(\TT_{\alpha,0}) \le \log \sup\{\Re \lambda: \lambda\in \sigma(\ee^{\tau\TT_{\alpha,0}})\} $, which is in turn upper bounded by $1- \frac5{4\alpha}$ by Lemma \ref{lem:spectrum-TTa0}. 
\end{proof}

From  \cite[Cor 2.11, p. 258]{zbMATH01354832}, it follows that for all $w>\omega_{\ess} (\Tab)$, $\sigma(\Tab)\cap \{\Re\lambda>w\}$  consists of only finitely many eigenvalues with finite multiplicity. In addition, for each $\delta>0$ and for each divergence-free $U_0\in L^2$, 
\begin{align}
	\|\ee^{\tau \Tab} U_0 \|_{L^2} \lesssim_\delta \ee^{\tau(a+\delta)} \|U_0\|_{L^2}.\label{e:most-unstable-eigenvalue}
\end{align}
We write $\lambda=a+b\ii$, $a>0$ for the largest eigenvalue of $\Tab$, with eigenfunction $\eta$. It follows that
\begin{align}
    \|\ee^{\tau \Tab}\eta  \|_{L^2} =e^{\tau a} \|\eta\|_{L^2}. \label{e:eigen-estimates}
\end{align} 
\begin{lem}[Parabolic regularity]	\label{lem:parabolic-reg}
For all $\sigma'\ge \sigma\ge 0$, $\delta>0$, and $\tau>0$,
\begin{align}
    \| \ee^{\tau \Tab} U_0\|_{H^{\sigma' }}\lesssim_{\sigma,\sigma',\delta} \frac{\ee^{\tau(a+\delta)} \|U_0\|_{H^{\sigma}}}{\tau^{(\sigma'-\sigma)/2\alpha }} .\label{e:parabolic-reg}
\end{align}
\end{lem}
\begin{proof}
For times $\tau \le1$ say, and $\sigma = 0$, these estimates are implied by the estimates in \eqref{e:local-estimates}. The estimates for $\sigma>0$ are similarly proven by differentiating the equation for $u(x,t) = \frac1{t^{1-1/2\alpha} } U(xt^{-1/2\alpha},\log t)$.

For large times, we first use the small time estimate to drop to $L^2$:
\[ \| \ee^{\tau \Tab} U_0\|_{H^{\sigma' }}\lesssim_{a,\sigma,\delta} \| \ee^{(\tau-1) \Tab} U_0\|_{L^2}. \]
It follows from \eqref{e:most-unstable-eigenvalue} that $\| \ee^{\tau \Tab} U_0\|_{H^{\sigma'}}  \lesssim_{a,\sigma,\delta} e^{\tau(a+\delta)} \|U_0\|_{L^2}$, which proves the lemma.
\end{proof}
\subsection{Nonuniqueness}\label{ss:nonunique}
For completeness, we repeat the argument laid out in \cite[\S4]{zbMATH07583008} to construct the second solution of \eqref{e:ns}.

We look for a solution $U$ to \eqref{e:ss-vel} that vanishes as $\tau\to -\infty$  with the ansatz $U=\beta\bar U + \Ulin + \Uper$, where $\Ulin\coloneq \Re(e^{\lambda\tau }\eta)$ solves the linear  equation $(\del_\tau - \Tab)\Ulin = 0$. Substituting this ansatz into \eqref{e:ss-vel}, we find that the perturbation $\Uper$ satisfies 
 the integral equation $\Uper = \mathcal T(\Uper) $, where
\begin{align*}
&\mathcal T(\Uper)(\tau) \coloneq \\ & -\int\limits_{-\infty}^\tau  \ee^{(\tau-s) \Tab} \mathbb P\Big(\Ulin\cdot\nabla\Ulin+ \Uper\cdot\nabla\Ulin+ \Ulin\cdot\nabla\Uper+ \Uper\cdot\nabla\Uper\Big) \dd s.     
\end{align*}
We construct $\Uper$ as a fixed point of $\mathcal T$ via the contraction mapping theorem. So, we set $N>5/2$, $\epsilon_0\in(0,a)$ and introduce the space 
\[  X \coloneq  X(N,T)\coloneq C^0((-\infty,T];H^N(\mathbb R^3;\mathbb R^3))\]
with the weighted norm ($\epsilon_0>0$ to be chosen momentarily)
\[ \|U\|_X \coloneq \sup_{\tau \le  T} \ee^{-(a+\epsilon_0) \tau} \|U(\tau)\|_{H^N}.\]
\begin{prop}\label{p:contraction}
 Let $B_X$ be the closed unit ball of $X$.
Then for $T$ sufficiently large and negative, and $N>5/2$,  $\mathcal T$ maps $B_X\to B_X$ and is a contraction.
\end{prop}
\begin{proof} Fix $\delta<a$ and $\epsilon_0<a$.
$\mathcal T$ splits into three terms $\mathcal T(U) = \mathcal B(U,U) + \mathcal LU + \mathcal G$, where
\begin{align}
     -\mathcal B ( U,V) &= \int_{-\infty}^\tau \ee^{(\tau-s)\Tab} \mathbb P(U\cdot\nabla V) \dd s,
     \\
     -\mathcal LU &= \int_{-\infty}^\tau \ee^{(\tau-s)\Tab} \mathbb P(\Ulin \cdot\nabla U + U\cdot\nabla \Ulin )\dd s,
     \\
     -\mathcal G &= \int_{-\infty}^\tau \ee^{(\tau-s)\Tab} \mathbb P(\Ulin \cdot\nabla \Ulin ) \dd s,
\end{align}
Recall that $H^{N-1}$ for $N>5/2$ is a Banach algebra. So, we have  the estimate $\|U\cdot\nabla V\|_{H^{N-1}} \le \|U\|_{H^N} \|V\|_{H^N}$. Combining this estimate with Lemma \ref{lem:parabolic-reg} gives the following crude estimate
\begin{align}
     &\|\mathcal B(U,U)(\tau) \|_{H^{N+(\alpha -1 /2)}} 
     \lesssim_{N,\delta} \int_{-\infty}^\tau  \frac{\ee^{(\tau-s)(a+\delta)}}{(\tau-s)^{(2\alpha+1)/4\alpha }} \|U\|_{H^N}^2 \dd s \notag 
    \\
    &\le \ee^{2\tau(a+\epsilon_0)}\|U\|^2_X\int_{-\infty}^\tau \frac{\ee^{(\tau-s)(\delta-a-2\epsilon_0)}}{(\tau-s)^{(2\alpha+1)/4\alpha }}  \dd s  \lesssim_{\epsilon_0,\delta} \ee^{2\tau(a+\epsilon_0)}\|U\|_X^2, \label{e:bootstrap-B}
\end{align}
as long as $\delta <a+2\epsilon_0$. Hence $\|\mathcal B(U,U)\|_X \lesssim e^{T(a+\epsilon_0)} \|U\|_X^2$. 

By replacing the estimate of $U$ in  Lemma \ref{lem:parabolic-reg} by the estimate \eqref{e:eigen-estimates} of $\Ulin$, we similarly obtain (using $\delta  < a$) the estimates for all $N>5/2$,
\begin{align}
    \|\mathcal LU(\tau)\|_{H^{N}} \lesssim_{N,\epsilon_0,\delta} \ee^{\tau(2a+\epsilon_0)} \|U\|_X,
    \quad 
     \|\mathcal G(\tau)\|_{H^{N}} \lesssim_{N,\epsilon_0,\delta} \ee^{2\tau a} 
     \label{e:bootstrap-LG},
\end{align}
which lead to the estimates
\begin{align}
    \|\mathcal LU\|_{X} \lesssim_{N,\epsilon_0,\delta} \ee^{T a} \|U\|_X,
\qquad      \|\mathcal G\|_X \lesssim_{N,\epsilon_0,\delta} \ee^{T (a-\epsilon_0) }. \notag 
\end{align}
It follows that by choosing $T$ sufficiently large and negative, we can make $\|\mathcal B\|$, $\|\mathcal L\|$ and $\|\mathcal G\|_X$ as small as we wish. It follows that for $T$ sufficiently large and negative, $\|\mathcal T(U)\|_X \le 1$ if $\|U\|_X\le 1$. In addition, since
\[ \|\mathcal T(U) - \mathcal T(V) \|_X \le (2\|\mathcal B\| + \|\mathcal L\|)\|U-V\|_X, \]
Choosing $T\ll -1$ also makes  $\mathcal T|_{B_X}$ contractive. This finishes the proof.
\end{proof}

Let $\Uper\in X(N,T)$ be the unique fixed point of $\mathcal T$ guaranteed by Proposition \ref{p:contraction}. The estimates \eqref{e:bootstrap-B} and \eqref{e:bootstrap-LG} let us bootstrap to $C^\infty$ regularity in space, and shows that $\Uper$ decays as $e^{2\tau a}$ as $\tau\to-\infty$. By construction, $U = \beta \bar U + \Ulin + \Uper$ solves \eqref{e:ss-vel}. Since $|\Ulin| \gtrsim e^{\tau a}$ as $\tau\to -\infty$, $\Uper \neq -\Ulin$, so $U \neq \beta\bar U$. Finally, undoing the similarity variable transform for both $U$ and $\beta \bar U$ gives a pair of distinct solutions of \eqref{e:ns}. 
\section{Acknowledgements}\label{s:acknowledgements}
We thank the anonymous referee and the associated editor for their invaluable comments
which helped to improve the paper. This work was supported by the National Key Research and
Development Program of China, No.2020YFA0712900, No 2022YFA1006700 and NSFC Grant 12071043.

\printbibliography
\end{document}